\documentclass[aps,preprint,nofootinbib,loatfix]{revtex4-1}

\usepackage{graphicx}
\usepackage{amsmath,amssymb,amsfonts,amsthm,amscd,bm}
\usepackage{booktabs}
\usepackage{listings}
\usepackage{xcolor,colortbl}
\usepackage{float}
\usepackage{kbordermatrix}

\def\beq{\begin{equation}}
\def\eeq{\end{equation}}

\def\Ei{{\rm Ei}}

\def\sharp{\#}
\def\vphi{\varphi}
\def\({\left(}
\def\){\right)}
\def\SL{SL_2\(\mathbb{Z}\)}

\makeatletter
\g@addto@macro\bfseries{\boldmath}
\makeatother

\makeatletter
\renewcommand*\env@matrix[1][\arraystretch]{%
  \edef\arraystretch{#1}%
  \hskip -\arraycolsep
  \let\@ifnextchar\new@ifnextchar
  \array{*\c@MaxMatrixCols c}}
\makeatother

\theoremstyle{plain}
\newtheorem{theorem}{Theorem}

\newtheorem{corollary}{Corollary}
\newtheorem{conjecture}{Conjecture}

\theoremstyle{remark}
\newtheorem{remark}{Remark}
\newtheorem{example}{Example}

\begin{document}

\title{Some Riemann Hypotheses from Random Walks over Primes}

\author{Guilherme Fran\c ca}
\email{guifranca@gmail.com}
\affiliation{Cornell University, Physics Department, Ithaca, NY 14850}
\affiliation{Johns Hopkins University, Baltimore, MD 21218}

\author{Andr\' e  LeClair}
\email{andre.leclair@gmail.com}
\affiliation{Cornell University,  Physics Department, Ithaca, NY 14850}

\begin{abstract}
The aim of this article  is to investigate how various
Riemann Hypotheses would
follow only from properties of the prime numbers. To this end,
we consider two classes of \mbox{$L$-functions}, namely,
non-principal Dirichlet and those based on cusp forms.
The simplest example of the latter is based on the Ramanujan tau arithmetic
function.
For both classes we prove that if a particular  trigonometric series involving
sums of
multiplicative characters over primes  is $O(\sqrt{N})$,  then the
Euler product converges in the right half of the critical strip.
When this result is combined with the functional equation,  the non-trivial
zeros are constrained to lie on the critical line.
We argue that this $\sqrt{N}$ growth is a consequence of the series
behaving like a one-dimensional random walk.
Based on these results we obtain an equation which relates every
individual non-trivial zero of the $L$-function to a sum involving all
the primes.
Finally,  we briefly mention
important differences
for principal Dirichlet $L$-functions  due to the existence of
the pole at $s=1$, in which the Riemann $\zeta$-function
is a particular case.
\end{abstract}

%\keywords{Euler Product; Riemann Hypothesis.}
%\ccode{Mathematics Subject Classification 2000: 11M26}

\maketitle

%%%%%%%%%%%%%%%%%%%%%%%%%%%%%%%%%%%%%%%%%%%%%%%%%%%%%%%%%%%%%%%%%%%%%%%%%%%%%%%
\section{Introduction}

Montgomery conjectured that
the pair correlation function between the ordinates of  the Riemann zeros
on the critical line  satisfy
the GUE statistics of random
matrix theory \cite{Montgomery}.  On the other hand, Riemann \cite{Riemann}
obtained an exact
formula for the
prime number counting function $\pi(x)$ in terms of the
non-trivial zeros of $\zeta(s)$.  This suggests that if the Riemann
Hypothesis is true, then this should imply some
kind of randomness of the primes. It has been remarked  by many  authors that
the primes appear random, and this is sometimes referred to as
pseudo-randomness of the primes \cite{Tao}.

In this article we address the following question, which is
effectively the reverse of the previous paragraph.
What kind of specific pseudo-randomness of the primes  would imply
the Riemann Hypothesis?
This requires a concrete characterization of the pseudo-randomness.
We provide such a characterization by arguing that certain deterministic
trigonometric sums over primes, involving  multiplicative functions,
behave like random walks,  namely grow as $\sqrt{N}$.
However, we are not able to fully prove this
$\sqrt{N}$ growth,   and thus  we will take it  as a conjecture.
This conjecture may appear to be reminiscent of Merten's false conjecture
that $|\sum_{n=1}^N  \mu(n) | < \sqrt{N}$,
where $\mu(n)$ is the M\"obius function.
However, it is different in an important manner:
our series involves a sum over primes rather than integers which in
some sense renders it more random.

The main result of this paper can be stated as follows.
Consider $L$-functions based on non-principal Dirichlet characters
and on cusp forms. We prove that, assuming the claim from the previous
paragraph concerning the random walk behavior,
the Euler product converges to the
right of the critical line.

This article is partly  based on the ideas in \cite{EulerProduct}
and is intended
to clarify it with more precise statements.
There is an important difference
between the cases mentioned above and  principal Dirichlet $L$-functions,
where $\zeta(s)$ is a particular case, and this is emphasized more here.
We will not consider this latter case in detail,  but
we briefly mention these subtleties in the last section of the paper.

%%%%%%%%%%%%%%%%%%%%%%%%%%%%%%%%%%%%%%%%%%%%%%%%%%%%%%%%%%%%%%%%%%%%%%%%%%%%%%%
\section{On the growth  of  series of Multiplicative Functions over primes}
\label{characters}

In this section we consider the asymptotic growth of certain
trigonometric sums over primes involving
multiplicative arithmetic functions. We propose that these sums
have the same growth as one-dimensional random walks.

Let $c(n)$ be a multiplicative function, i.e. $c(1) = 1$ and
$c(m n) = c(m) c(n)$ if $m$ and $n$ are coprime integers,
and let $p$ denote an
arbitrary prime number. We can always write
%\beq\label{angles}
$c(p) = |c(p)| e^{i \theta_p}$.
%\eeq
Now consider the trigonometric sum
\beq\label{CN}
C_N = \sum_{n=1}^{N} \cos \theta_{p_n}
\eeq
where $p_n$ denotes the $n$th prime; $p_1 = 2$, $p_2=3$,
and so forth. We wish to estimate
the size of this sum,  specifically  how its growth depends on $N$.

\subsection{Non-Principal Dirichlet Characters}

\subsubsection{The Main Conjecture}

Now let $c(n)=\chi(n)$ be a Dirichlet character modulo $k$,  where $k$ is
a positive integer.
The function $\chi$ is completely  multiplicative, i.e. $\chi(1) = 1$,
$\chi( m n) = \chi(m) \chi(n)$ for all $m, n \in \mathbb{Z}^{+}$, and
obeys the periodicity $\chi(n) = \chi(n+k)$. Its values are either
$\chi(n) = 0$, or $|\chi(n)| = 1$ if and only
if $n$ is coprime to $k$.
For a given $k$ there are $\varphi(k)$ different characters
which can be labeled as
$\{\chi_1, \dotsc, \chi_{\varphi(k)}\}$. The arithmetic
function $\varphi(k)$ is the Euler totient.   We will  omit the index of
the character except for $\chi_1$ which denotes
the \emph{principal} character, defined as $\chi_1(n) = 1$ if $n$ is coprime
to $k$, and $\chi_1(n)=0$ otherwise.
The Riemann $\zeta$-function corresponds  to the trivial principal character
with $k=1$.

For a non-principal character
the non-zero elements correspond to $\varphi(k)$-th roots
of unity given by
$\chi(n) \equiv e^{i\theta_n} = e^{2\pi i \nu_n / \varphi(k)}$
for some $\nu_n \in \mathbb{Z}$.
The \emph{distinct} phases of these roots of unity form a discrete and finite
set denoted by
\beq\label{PhiDef}
\theta_n \in \Phi \equiv \left\{ \phi_1, \phi_2, \dotsc, \phi_r \right\},
\qquad \mbox{where $r \leq \varphi(k)$.}
\eeq
Here $r$  is the order of the character.
For $k$ prime,  $r = \varphi (k) = k-1$.

For our purposes,
there is an important distinction between principal verses non-principal
characters.
The principal characters  satisfy
\beq \label{prin}
\sum_{n=1}^{k-1}  \chi_{1} (n)  =   \varphi(k) \neq 0,
\eeq
while non-principal characters satisfy
\beq
\label{chisum}
\sum_{n=1}^{k-1}  \chi(n) = 0.
\eeq
The above relation \eqref{chisum} shows that the angles in $\Phi$ are equally
spaced over the unit circle for non-principal characters.
On the other hand, this is not the case for
principal characters due to \eqref{prin};
in fact the angles $\theta_n$ are all zero.

For the sake of clarity, let us now simply state the main hypothesis
that the remainder of this work relies  upon.
We cannot prove this conjecture,   however we
will subsequently provide supporting,  although heuristic,  arguments.

\begin{conjecture}\label{ConjDir}
Let $p_n$ be the $n$th prime and $\chi(p_n) = e^{i\theta_{p_n}} \neq 0$
the value
of a non-principal Dirichlet character modulo $k$. Consider  the series
\beq
\label{C_dir}
C_N = \sum_{\substack{n=1 \\ p_n \nmid k}}^{N} \cos \theta_{p_n}.
\eeq
Then $C_N = O(\sqrt{N})$ as $N\to\infty$, up to logs.
By the latter we mean, for instance,
$C_N = O(\sqrt{N}  \log^a N)$ for any  positive power $a$,
or $C_N = O(\sqrt{N  \log \log N})$,  etc.  suffices.
\end{conjecture}

The main supporting argument is an  analogy with one-dimensional random
walks,  which  are known to grow as $\sqrt{N}$.
Although the series $C_N$ is completely deterministic,
its random aspect stems from the pseudo-randomness
of the primes,  which is largely a consequence of
their multiplicative independence.
The event of an integer being divisible by a prime $p$
and also divisible by a different prime $q$ are mutually independent.
A  simple argument is Kac's heuristic \cite{Kac}:
let $P_m(n)$ denote the probability that an integer $m$ is  divisible by
$n$. The probability that $m$ is even, i.e. divisible by $2$,
is $P_m(2) = 1/2$. Similarly, $P_m(n) = 1/n$.
We therefore have $P_m(p \, q) = (p\, q)^{-1} = P_m(p) P_m(q)$,
and the events are independent.
Because of the multiplicative property of $c(n)$
this independence of the primes  extends  to quantities involving
$c(p)$,   in that  $c(p)$  is  independent of $c(q)$ for primes
$p \neq q$.
Moreover, if $\{\theta_{p}\}$ are equidistributed over a finite set of
possible angles, then the deterministic sum \eqref{CN} is expected to behave
like a random walk since each term $\cos\theta_p$ mimics an independent
and identically distributed (iid) random variable.
Analogously, if we build a random model capturing
the main features of \eqref{CN} it should provide an accurate description
of some of its important  global properties.

Let us provide a more detailed argument.
First a theorem of Dirichlet addresses the identically distributed aspect.

\begin{theorem}
[Dirichlet] \label{prob_angle}
Let $\chi(n)=e^{i\theta_n} \ne 0$  be a non-principal
Dirichlet character modulo $k$ and $\pi (x)$  the number of
primes less than $x$.
These distinct roots of unity form a finite and discrete set,
$\theta_n \in \Phi = \{ \phi_1,\phi_2, \dotsc, \phi_{r}\}$ with
$r \leq \varphi(k)$.  Then for a prime $p$ we have
\beq \label{RW1}
F(\theta_p=\phi_i) = \lim_{x \to \infty} \dfrac{
\sharp \left\{ p \leq x:  \, \theta_p = \phi_i \right\}}{\pi(x)}
=  \dfrac{1}{r}
\eeq
for all $i=1,2,\dotsc,r$,
where $F(\theta_p=\phi_i)$ denotes the frequency of the event
$\theta_p=\phi_i$ occurring.
\end{theorem}
\begin{proof}
Let $[a_i]$   denote the residue classes modulo $k$
for $a_i$ and $k$ coprime,
namely the set of integers  $[a_i] = \{ a_i  \ \text{mod}~k \}$.
There are $\varphi (k)$  independent classes
and they form a group.
Of these  classes let the set of integers  $[ a_i ]$  denote the particular
residue class where   $\chi (a_i \ \text{mod}~k)  =  e^{i \phi_i}$.
Then
\beq
F(\theta_p =  \phi_i )  =   F(p =  a_i  \ \text{mod}~k).
\eeq
Dirichlet's theorem states that there are an infinite number of primes
in arithmetic progressions,
and $F(p=a_i) = 1/r $  independent of $a_i$.   In particular,
\beq
\label{DirDave}
\pi (x, a, k ) =        \sharp \left\{  p< x, p \equiv a \ \text{mod}~k  \quad
\text{with} \,  (a,k) = 1 \right\}   =   \frac{ \pi (x) }{\vphi (k) }
\eeq
in the limit $x \to \infty$.
(See for instance \cite[Chap. 22]{Davenport}.) 
\end{proof}

The frequencies $F$ can be interpreted as probabilities,
however we will continue to refer to them
as frequencies.
Next consider  the joint frequency,   defined by
\beq \label{RW2}
F(\theta_p = \phi_i, \theta_q = \phi_j ) =
\lim_{x \to \infty} \dfrac{ \sharp \left\{ p,q \leq x : \,
\theta_p = \phi_i \ \mbox{and} \ \theta_q = \phi_j \right\}}{\pi^2(x)}
\eeq
for all $i,j=1,2,\dotsc,r$.
The events $p=a_i $ and $q=a_j $   ($ {\rm mod} ~ k)$  are independent due to
the multiplicative independence of the primes.
Thus one expects
\beq
\label{CondP}
\begin{split}
F( \theta_p = \phi_i , \theta_q = \phi_j ) &=
F(\theta_p = \phi_i |  \theta_q = \phi_j )  F(\theta_q = \phi_j)  \\
&= F(\theta_p = \phi_i ) F (\theta_q = \phi_j ) = \dfrac{1}{r^2}.
\end{split}
\eeq

In other words, for a randomly chosen prime,
each angle $\phi_i \in \Phi$ is equally
likely to  be the value of $\theta_p$,  i.e.,
$\theta_p$ is uniformly distributed over $\Phi$.
Moreover, $\theta_p$ and $\theta_q$ are independent.
Thus the series \eqref{C_dir} should  behave like a random walk,
and this is the  primary motivation for Conjecture~\ref{ConjDir}.

\begin{remark}
\label{remarkCram}
In \cite{ALCLT}  one of us studied a probabilistic
model for $C_N$ and proved a central limit theorem for it.
Namely,  in the definition of $C_N$,
the primes $\{p_1, p_2, \ldots\}$ were replaced with
$\{p'_1,  p'_2, \ldots \}$ where the $p'$ were chosen according
to Cram\'er's random model for the primes \cite{Cramer}.
$C_N /\sqrt{N}$ is now a random variable with a probability distribution,
which we showed to be a normal distribution
as $N\to \infty$.
The latter implies $C_N = O(N^{1/2 + \epsilon})$ for any $\epsilon > 0$ with
probability equal to $1$.
Also,  the law of iterated logarithm
implies $C_N = O(\sqrt{N\log \log N})$,
which as stated in Conjecture \ref{ConjDir},  will be sufficient for
our purposes.
\end{remark}

\subsubsection{Numerical Evidence}
\label{example_probs}

Let us also provide numerical evidence for the above statements.
In Figure~\ref{fig:prob}a we have an example with $k=7$.
The specific character is
\beq
\label{chi7}
\{ \chi(1), \dotsc, \chi(7) \} =
\{ 1, e^{ 2\pi i /3},e^{\pi i / 3}, e^{-2\pi i / 3}, e^{-\pi i / 3}, -1,0\}
\eeq
with $r  = 6$,
so  that $F(\theta_p=\phi_i) = 1/6 = 0.1666\dotsm$.
This table was computed with $x=10^9$ in \eqref{RW1}.
One can see the equally spaced angles over the unit circle, and the numerical
results verify that $\theta_p$ is uniformly distributed over $\Phi$, as
stated  in Theorem \ref{prob_angle}.

Let us also check \eqref{RW2}. All the joint
frequencies  are shown in the following matrix:
\beq 
%\nonumber
F(\phi_i,\phi_j) = 10^{-2}\times \kbordermatrix{
& \phi_1 & \phi_2 & \phi_3 & \phi_4 & \phi_5 & \phi_6 \\
\phi_1
&2.7293
&2.7679
&2.7454
&2.7518
&2.7583
&2.7647 \\
\phi_2
&
&2.8071
&2.7842
&2.7908
&2.7973
&2.8038 \\
\phi_3
&
&
&2.7616
&2.7680
&2.7745
&2.7810 \\
\phi_4
&
&
&
&2.7745
&2.7810
&2.7875 \\
\phi_5
&
&
&
&
&2.7875
&2.7940 \\
\phi_6
&
&
&
&
&
&2.8006
}.
\eeq
Here we used only $x=5 \times 10^4$. These values are all  close
to the predicted theoretical value $F(\phi_i,\phi_j) = 
(1/6)^2 = 2.7777\dotsm \times 10^{-2}$, and get even closer
with higher $x$.

\begin{figure}
\begin{minipage}{.5\textwidth}
\centering
\vspace{.5em}
\begin{minipage}{.45\textwidth}
\includegraphics[width=\textwidth]{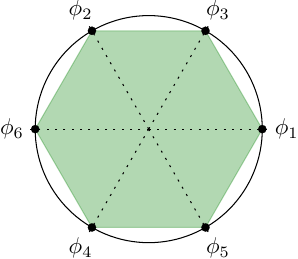}
\end{minipage}
\begin{minipage}{.5\textwidth}
\footnotesize
\flushleft\hspace{1em}
\renewcommand{\arraystretch}{.65}
\begin{tabular}{@{}ccc@{}} \toprule[1pt]
$n$ & $\phi_n$ & $F(\theta_p=\phi_n)$ \\ \midrule[.5pt]
$1$ & $0$ & $0.1666594293$ \\
$2$ & $2\pi/3$ & $0.1666707377$ \\
$3$ & $\pi/3$ & $0.1666771687$ \\
$4$ & $-2\pi/3$ & $0.1666554960$ \\
$5$ & $-\pi/3$ & $0.1666674730$ \\
$6$ & $\pi$ & $0.1666696757$ \\ \midrule[.5pt]
\multicolumn{3}{c}{$\sum_{n} F_n = 0.9999999803$} \\ \bottomrule[1pt]
\end{tabular}
\end{minipage}\\[1em]
(a)
\end{minipage}
\begin{minipage}{.49\textwidth}
\centering
\hspace{1em}\includegraphics[width=.9\textwidth]{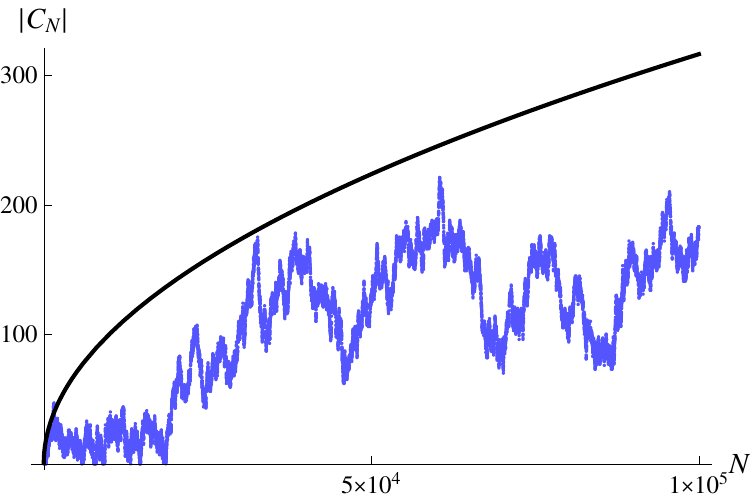} \\
\hspace{1em}(b)
\end{minipage}
\caption{\label{fig:prob} (a) Numerical verification of \eqref{RW1} with the
$k=7$ character in \eqref{chi7}.
Notice the equally spaced angles over the unit circle, and
the corresponding probabilities shown in the table. We use $x=10^9$.
(b) Numerical evidence for  \eqref{C_dir} (blue dots) with a $k=19$ character
shown in \eqref{chi19}.
We compare with the curve
$\sqrt{N}$ (solid black line).
}
\end{figure}

In Figure~\ref{fig:prob}b we provide evidence  that
\eqref{C_dir} is $O(\sqrt{N})$. Here we choose a Dirichlet character
with modulus $k=19$ as shown below:
\begin{equation}\label{chi19}
\{ \chi(1), \dotsc, \chi(19) \} = \{  
1, -1, -1, 1, 1, 1, 1, -1, 1, -1, 1, -1, -1, -1, -1, 1, 1, -1, 0
\}.
\end{equation}
The blue dots represent
the series \eqref{C_dir} and the solid black curve is $\sqrt{N}$.

\subsection{Fourier Coefficients of Cusp Forms}

Let us extend the above arguments to the Fourier coefficients
of cusp forms. We will briefly review the general Hecke theory
in Section~\ref{sec:modular} where we will explain the significance of
being a cusp form. The simplest and best-known example
is the weight $k=12$ modular form,  which is
the $24$th power of the Dedekind $\eta$-function
\beq \label{discriminant}
\Delta(z) = q \prod_{n=1}^\infty (1-q^n)^{24} =
\sum_{n=1}^\infty   \tau (n)  q^n ,
\eeq
where $q = e^{2\pi i z}$ and $\Im(z) > 0$.
Here  the  Fourier coefficients $\tau: \mathbb{N} \to \mathbb{Z}$
are known as the Ramanujan $\tau$ arithmetic function.
More generally, let us refer to the Fourier coefficients
of cusp forms as $c(n)$, where $c:\mathbb{N} \to \mathbb{R}$ is a
multiplicative function. Thus we can write
\beq
c(n) = |c(n)| \cos \theta_n, \qquad \mbox{with $\theta_n \in \{0, \pi\}$}.
\eeq
The series to consider is now
\beq \label{tau2}
C_N =  \sum_{n=1}^N \cos \theta_{p_n},
\qquad \mbox{ with $\cos \theta_{p_n} = \pm 1$},
\eeq
and resembles even more closely the original discrete random walk.
Let us assume that
Theorem  \ref{prob_angle} holds in the same way but now we
have $\Phi = \{ 0, \pi \}$,
and then $F(\theta_p = 0) = F(\theta_p = \pi) = 1/2$.
For any two primes $p$ and $q$ these two events are independent.
As a consequence we have the analog
of  Conjecture \ref{ConjDir}.

\begin{conjecture}\label{ConjCusp}
The sum \eqref{tau2} obeys the bound $C_N = O(\sqrt{N})$ as $N\to\infty$,
up possibly to logs, i.e. $C_N = O(\sqrt{N} \log^a N)$ for any $a > 0$,
or $C_N = O(\sqrt{N \log \log N})$, etc.
\end{conjecture}

\begin{remark}
Deligne \cite{Deligne} proved that $|c(p)| \le 2 p^{(k-1)/2}$. This implies
that we can write
\beq\label{Frobenius}
c(p) =  2 p^{(k-1)/2}  \cos \alpha_p
\eeq
where $\alpha_p$ is called a Frobenius angle.
The aspect of a uniform distribution in Theorem  \ref{prob_angle}
can be seen as a weaker form of Sato-Tate
conjecture \cite{SatoTate}. Whereas
\eqref{tau2} only
concerns the signs $\cos\theta_p = \pm 1$, Sato-Tate
is much more specific. It asserts that $\alpha_p$ in \eqref{Frobenius}
is uniformly distributed over $[0, \pi]$ according to the function
$\tfrac{2}{\pi} \sin^2 \beta$, with $\beta \in [0,\pi]$.
Thus, Sato-Tate conjecture would imply our assumption
that the signs of $c(p)$ are equally likely
to be \mbox{$+1$ or $-1$}.
\end{remark}

%%%%%%%%%%%%%%%%%%%%%%%%%%%%%%%%%%%%%%%%%%%%%%%%%%%%%%%%%%%%%%%%%%%%%%%%%%%%%%%
\section{Convergence of the Euler Product for
Non-Principal Dirichlet $L$-Functions}
\label{sec:convergence}

Let $s=\sigma + i t$ be a complex variable. Given a Dirichlet
character $\chi$ modulo $k$ we have the Dirichlet $L$-series
\beq
\label{Lfunct}
L(s, \chi) = \sum_{n=1}^\infty\dfrac{\chi (n)}{n^s}.
\eeq
The domain of convergence of
Dirichlet series are always half-planes. Such series converge
absolutely for $\sigma > \sigma_a$,  where $\sigma_a$ is referred to
as the abscissa of  absolute convergence. There is also an
abscissa of convergence $\sigma_c \leq \sigma_a $.
For all Dirichlet series \eqref{Lfunct} we have $\sigma_a = 1$.
The analytic properties of \eqref{Lfunct} are as follows. If $\chi$
is non-principal, then $L(s,\chi)$ is analytic in the half-plane $\sigma>0$,
with no poles. If $\chi=\chi_1$ is principal, then
$L(s,\chi_1)$ has a simple pole at $s=1$, but it is analytic everywhere
else.
In this case $\sigma_c = \sigma_a = 1$.
There is a functional equation relating $L(s,\chi)$ to
$\overline{L(1-\bar{s},\chi)}$,
thus the critical line is $\sigma =1/2$;
the critical strip is  the region $0\leq  \sigma \leq 1$ where
all  the non-trivial zeros lie.  From now on we consider only
non-principal characters, $\chi \ne \chi_1$.

Since there is no pole at $s=1$,
it is possible that $\sigma_c < \sigma_a$.
In fact $\sigma_c =0$. This is easy to see from the
Dirichlet's convergence test \cite[pp. 17]{Whittaker}.
Set $t=0$ and write $L(s,\chi) = \sum_n  \chi (n) \, \ell_n $
where $\ell_n = 1/n^\sigma$.
One has $\ell_n > \ell_{n+1}$ and $\lim_{n\to \infty}  \ell_n = 0$, if
$\sigma>0$.      Now,  due to \eqref{chisum},
$\big|  \sum_{n=1}^N \chi (n) \big|  \leq c$
%\beq
%\label{convg1}
%\left|  \sum_{n=1}^N \chi (n) \right|  \leq c   
%\eeq  
for every integer $N$ and for some constant $c$.
In fact,
$c  =  \max_j  \big\{  \sum_{n=1}^j \chi (n) \big\}$ for
$j=1, 2,\dotsc, k-2$.
%\beq
%\label{Cconv}
%c  =  \max  \left\{  \sum_{n=1}^j \chi (n),\ j=1, 2, \dotsc, k-2   \right\}. 
%\eeq
Thus, since convergence of Dirichlet series are always half-planes,
the series \eqref{Lfunct} converges for all complex $s$ with $\Re (s) >0$.

Due to the completely  multiplicative property of $\chi$
one has the Euler product formula
\beq
\label{EPFDir}
L(s, \chi )  =  \prod_{n=1}^\infty \( 1- \dfrac{\chi (p_n)}{p_n^{\,s}} \)^{-1}.
\eeq
Because \eqref{Lfunct} converges for $\Re(s)>0$ this opens
up the possibility that the product in
\eqref{EPFDir}
converges for $\Re (s) > \sigma_c$, for some $\sigma_c >0$, where now
we refer specifically to the abscissa of convergence of the product
in \eqref{EPFDir}.
We will argue that in this case
$\sigma_c = 1/2$ and
the equality between \eqref{EPFDir} and \eqref{Lfunct}
is valid for $\Re (s) > 1/2$ since
both sides of the equation converge in this region.

Taking the formal logarithm on both sides of \eqref{EPFDir},
and assuming the principal
branch, we have
$\log L(s,\chi) =  X(s,\chi) + R(s,\chi)$ where
%\beq \label{logsum}  
%\log L(s,\chi) =  
%- \sum_{n=1}^\infty \log \( 1- \dfrac{\chi (p_n) }{p_n^{\,s}} \)  = 
%X(s,\chi) + R(s,\chi)
%\eeq
%where 
\beq\label{PDir}
X(s,\chi) = \sum_{n=1}^{\infty} \dfrac{\chi(p_n)}{p_n^{\, s}}, \qquad
R(s,\chi) = \sum_{n=1}^{\infty} \sum_{m=2}^{\infty}
\dfrac{\chi(p_n)^{m}}{m p_n^{\, ms}}.
\eeq
Now $R(s,\chi)$ absolutely converges for $\sigma > 1/2$, therefore
\beq\label{logsum2}
\log L(s,\chi) = X(s,\chi) + O(1)
\eeq
and convergence of the Euler product to the right of the critical line
depends only on $X(s,\chi)$.
The next result shows that Conjecture~\ref{ConjDir} is sufficient
to ensure that $X(s, \chi)$ also converges in the half-plane $\sigma > 1/2$.

\begin{theorem}\label{DirConvergence}
Let $L(s,\chi)$ be a non-principal Dirichlet $L$-function.
Assuming Conjecture~\ref{ConjDir}, the %generalized 
Dirichlet series $X( s, \chi)$ defined in
\eqref{PDir} has abscissa of convergence $\sigma_c = 1/2$.
This implies that the Euler product \eqref{EPFDir} also has the
half-plane of convergence given by $\sigma > 1/2$.
\end{theorem}
\begin{proof}
Analogously to \eqref{C_dir}, let us define
\begin{equation}
\label{CxDef}
C(x) = \sum_{p \le x} \cos \theta_p.
\end{equation}
It is sufficient to consider the real part of $X(s, \chi)$ in \eqref{PDir}
with $t=0$, i.e.
\begin{equation}
\label{SDef}
S(\sigma, \chi) = \sum_{n=1}^\infty \dfrac{\cos \theta_{p_n}}{p_n^{\, \sigma}}.
\end{equation}
Notice that $\cos \theta_{p_n} = C(p_n) - C(p_{n-1})$, hence
\begin{equation}
S(\sigma, \chi) = \sum_{n=1}^{\infty} C(p_n) \left( \dfrac{1}{p_n^{\,\sigma}} - 
\dfrac{1}{p_{n+1}^{\,\sigma}}\right) = \sigma \sum_{n=1}^{\infty} 
C(p_n) \int_{p_n}^{p_{n+1}} \dfrac{1}{u^{\sigma+1}} du.
\end{equation}
Since $C(u) = C(p_n)$ is a constant for $u \in (p_n, p_{n+1})$, we can write
\begin{equation}
S(\sigma, \chi) = \sigma \int_{2}^{\infty} \dfrac{C(u)}{u^{\sigma+1}} du.
\end{equation}
With $C(u)$ obeying Conjecture~\ref{ConjDir} the above integral is finite,
provided $\sigma > \sigma_c = 1/2$, implying that \eqref{SDef} has
abscissa of convergence $\sigma_c = 1/2$.
For instance, assuming
$C(x) = O\left(\sqrt{x} \log^a x \right)$ yields
\begin{equation}
|S(\sigma, \chi)| \le  K \int_{1}^\infty 
\dfrac{\log^a u}{u^{\sigma+1/2}} du = 
K \dfrac{\Gamma(a+1)}{(\sigma - 1/2)^{a+1}}
\end{equation}
for some constant $K > 0$, and with $\sigma > \sigma_c = 1/2$. The above
integral diverges if $\sigma \le \sigma_c$.

Since convergence of Dirichlet series are always half-planes, the above
result implies that the real part of $X(s, \chi)$ converges
for any complex $s = \sigma + it$ with $\sigma > 1/2$. Analogous argument
applies to the imaginary part of $X(s, \chi)$, and the proof is complete.
\end{proof}

Compelling numerical evidence for Theorem~\ref{DirConvergence}  has
already
been given in \cite{EulerProduct}.

\begin{remark}
The goal of this work was to obtain a result that is unconditional
on the Riemann Hypothesis.
On the other hand,  if one assumes the generalized Riemann hypothesis,
then  one can indeed show that
$C(x) = O(\sqrt{x} \log^2 x )$ \cite{Davenport},
and thereby conclude the Euler product converges by Theorem
\ref{DirConvergence}.
\end{remark}

\begin{corollary}\label{DirRH}
If Conjecture~\ref{ConjDir} is true unconditionally,
then Theorem~\ref{DirConvergence} is also true unconditionally, implying that
all non-trivial zeros of a non-principal Dirichlet $L$-function must
be on the critical line $\sigma = 1/2$, which is the (generalized)
Riemann Hypothesis.
\end{corollary}
\begin{proof} The argument is very simple, analogous to
showing that there are no zeros with $\sigma > 1$.
From \eqref{logsum2},
a zero $\rho$ of $L(\rho, \chi) = 0$ requires $X(\rho, \chi) \to -\infty$.
If $X(s,\chi)$ converges for $\sigma > 1/2$ there are no zeros of
$L(s,\chi)$ in this region. From the functional equation, which is
a symmetry between $L(s,\chi)$ and
$\overline{L(1-\bar{s}, \chi )}$, it implies no non-trivial
zeros with $\sigma < 1/2$. Since it is known that there are infinite zeros
in the critical strip $0 < \sigma < 1$, they must all be on
the line $\sigma = 1/2$.
\end{proof}

%%%%%%%%%%%%%%%%%%%%%%%%%%%%%%%%%%%%%%%%%%%%%%%%%%%%%%%%%%%%%%%%%%%%%%%%%%%%%%%
\section{Convergence of the Euler Product for $L$-functions
based on Cusp Forms}
\label{sec:modular}

In this section  we show that the same reasoning applies to $L$-functions based
on cusp forms.
Let $f(z)$ denote such a modular form  of weight $k$.
The $\SL$ transformations  imply the periodicity $f(z +1) = f(z )$,
thus it has a Fourier series
\beq \label{Fourier}
f(z ) = \sum_{n=0}^\infty  c(n) \, q^n  ,
\qquad q \equiv  e^{2 \pi i z}.
\eeq
If $c(0) = 0$ then $f$ is called a cusp form.
We will only consider  level-$1$,   entire modular forms.
From the Fourier coefficients \eqref{Fourier}
one can define the Dirichlet series
\beq \label{Lmod}
L\(s, f\) = \sum_{n=1}^\infty \dfrac{c\(n\)}{n^s}.
\eeq
For cusp forms  $L(s,f)$ is entire, i.e. has no poles.
It is also known that
\beq \label{growthmod}
c(n) = O(n^k),
\eeq
The validity of a Riemann Hypothesis for $L(s, f)$ remains
an open question. However, it was conjectured to be true by Ramanujan for
the $L$-function based on $c(n) = \tau (n)$.

The Fourier coefficients now  have the following multiplicative property:
$c(m) c(n)  = \hspace{-.4em}  
\sum_{d \,| (m,n)} d^{\,k-1} c\(\dfrac{mn}{d^2} \)$
where $d\, |(m,n)$ are the divisors of $(m,n)$.
This multiplicative property changes the form of the Euler product slightly in
comparison to the completely multiplicative case of Dirichlet $\chi$. We now
have
\beq \label{EPFmod}
L(s, f)  =  \prod_{n=1}^\infty \(  1 -  \dfrac{c(p_n)}{p_n^{\,s}}  
+ \dfrac{1}{p_n^{\,2s-k+1}}  \)^{-1}
\eeq
which  converges absolutely  for $\sigma > k/2+1$.

The analysis of Section~\ref{sec:convergence} extends
straightforwardly if one uses a non-trivial result of Deligne.
For the arithmetic function $\tau (n)$,
Ramanujan conjectured that it actually grows more slowly than
\eqref{growthmod},  namely  $\tau (p) = O(p^{11/2})$.
This was only proved in 1974 by Deligne \cite{Deligne} as a consequence of
his proof of the Weil conjectures.

\begin{theorem}[Deligne] \label{DeligneTheo}
For cusp forms we have
\beq \label{Deligne}
|c(p)| \leq  2  p^{(k-1)/2}.
\eeq
\end{theorem}

This theorem implies that \eqref{Lmod} converges
absolutely for $\sigma >  (k+1)/2 $.
Based on the  known functional equation,
the critical line is $\sigma = k/2$,
and the critical strip is the region
$(k-1)/2 \leq \sigma \leq (k+1)/2$.

\begin{theorem} \label{CuspConvergence}
Let $L(s,f)$ be an $L$-function based on a cusp
form $f$, with Euler product given by \eqref{EPFmod}.
Assuming Conjecture~\ref{ConjCusp}, the Euler product
converges to the right of the critical line $\sigma > k/2$.
\end{theorem}
\begin{proof}
The arguments are nearly the same as in Theorem~\ref{DirConvergence}.
Taking the logarithm of \eqref{EPFmod} one has $\log L(s, f) = X(s,f) + O(1)$,
where $O(1)$ denotes absolutely convergent terms  for $\sigma > k/2$ which
is a consequence of \eqref{Deligne}. Here we have
\beq\label{XMod}
X(s,f) = \sum_{n=1}^{\infty} \dfrac{c(p_n)}{p_n^{\, s}}.
\eeq
Therefore, all relies on the region of convergence of \eqref{XMod}. Without
loss of generality let us set $t=0$. Define
\beq\label{Btilde}
B(x) = \sum_{p \le x} |c(p)| \cos \theta_{p}.
\eeq
We thus have
\begin{equation} \label{CuspInt}
X(\sigma, f) = \sigma \int_{2}^{\infty} \dfrac{B(u)}{u^{\sigma+1}} du
\end{equation}
where we used $c(p_n) = |c(p_n)| \cos \theta_{p_n} = B(p_n) - B(p_{n-1})$,
and the fact that $B(u) = B(p_n)$ is a constant on the interval
$u \in (p_n, p_{n+1})$.

Using summation by parts on \eqref{Btilde}, together with
Deligne's bound \eqref{Deligne}, we conclude that
\begin{equation}
|B(x)| \le K  x^{k/2 - 1/2 } |C(x)| 
\end{equation}
for some constant $K > 0$, and
where
$C(x) = \sum_{p\le x} \cos\theta_p$ as in \eqref{tau2}.
Assuming Conjecture~\ref{ConjCusp} implies that \eqref{CuspInt} is
finite as long as $\sigma > k/2$. In particular, let
$C(x) = O\left( \sqrt{x} \log^a x\right)$ for any $a > 0$. Computing the
integral in \eqref{CuspInt} gives
$|X(\sigma, \chi)| \sim \Gamma(a+1) (\sigma - k/2)^{-a-1}$ if $\sigma >
k/2$, which is finite, and the integral diverges otherwise.
Since convergence of \eqref{XMod} must be a half-plane,
it converges for any complex $s = \sigma + it$ with
$\sigma > k/2$, hence the Euler product \eqref{EPFmod} also converges in
this region, as claimed.
\end{proof}

\begin{corollary} \label{CuspRH}
If Conjecture~\ref{ConjCusp} is true unconditionally then
Theorem~\ref{CuspConvergence} follows, and
the Riemann Hypothesis is true for $L$-functions based on cusp forms.
\end{corollary}
\begin{proof}
This simple argument is the same as in Corollary~\ref{DirRH}.
Convergence of $X(s, f)$ on $\sigma > k/2$ does not allow zeros
in this region. The functional equation which relates $L(s,f)$ to $L(k-s,f)$
forces the zeros to lie on the critical line $\sigma = k/2$, since it
excludes them from the left half of the critical strip $\sigma < k/2$.
\end{proof}

\begin{example}
Consider the Euler product based on the Ramanujan
$\tau$ function which is a weight $k=12$ cusp form:
\beq
\label{CPN}
L(s, \tau) = \lim_{N\to \infty}  P_N (s, \tau), \qquad
P_N(s, \tau) = \prod_{n=1}^N  \(  1 -  \dfrac{\tau(p_n)}{p_n^{\, s}}  
+ \dfrac{1}{p_n^{\,2s-11}}  \)^{-1}.
\eeq
Here $L(s, \tau)$ is the $L$-function, the analytic continuation of
the $L$-series. It is straightforward to numerically verify the above
results as shown in Figure~\ref{fig:ram}.
We expect the results to be more accurate as we increase $N$.
We choose $t=0$ in the rightmost column to explicitly show that there are no
divergences on the real line.

\begin{figure}
\begin{minipage}{.47\textwidth}
\footnotesize
\def\arraystretch{0.65}
\begin{tabular}{@{}c@{}c@{}c@{}}
\toprule[1pt]
$N$ & $\left| P_N(s,\tau) \right|$ & $\left| P_N(s,\tau) \right|$ \\
\midrule[.5pt]
$1\cdot10^1$  & $0.3085$ & $0.7545$  \\
$1\cdot10^2$  & $0.2291$ & $0.7747$  \\
$1\cdot10^3$  & $0.2530$ & $0.7992$  \\
$1\cdot10^4$  & $0.2548$ & $0.8104$  \\
$2\cdot10^4$  & $0.2562$ & $0.8114$  \\
$3\cdot10^4$  & $0.2586$ & $0.8133$  \\
$4\cdot10^4$  & $0.2597$ & $0.8142$  \\
%\midrule[.5pt]
\midrule[0.5pt]
&
\begin{tabular}{@{}l@{}}
$|L(s,\tau)| = 0.2610~~ $ \\
$s=6.25+100\, i$ 
\end{tabular} &
\begin{tabular}{@{}l@{}}
$|L(s,\tau)| = 0.8170\ $ \\
$s=6.25+0\,i$  
\end{tabular} \\ \bottomrule[1pt]
\end{tabular}
\end{minipage}
\begin{minipage}{.52\textwidth}
\includegraphics[width=\textwidth]{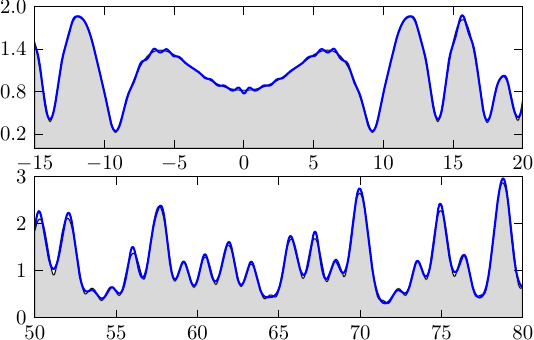}
\end{minipage}
\caption{\label{fig:ram}
Numerical results for the Euler product based on Ramanujan $\tau$.
The table show some values for two different points  inside the critical
strip as we increase
$N$. Note that there are no divergences when $t=0$.
We also have two plots for $s=6.25+i\,t$. In the first case
$t\in[-15,20]$, and in the second case $t\in[50,80]$. We choose
$N=10^2$ since for higher $N$ the curves are indistinguishable.
The (shaded) black line is $|L(s,\tau)|$ and the
blue line $|P_N(s,\tau)|$, against $t$.
}
\end{figure}
\end{example}

%%%%%%%%%%%%%%%%%%%%%%%%%%%%%%%%%%%%%%%%%%%%%%%%%%%%%%%%%%%%%%%%%%%%%%%%%%%%%%%
\section{An Equation Relating Non-Trivial Zeros and Primes}

Riemann \cite{Riemann} proved that the non-trivial
zeros of $\zeta(s)$ dictate the distribution of the primes:
by summing over the infinite number of
zeros
one can reconstruct the prime number counting function $\pi(x)$ exactly.
Assuming the Euler product converges to the right of the critical
line we can actually establish a  converse.

One can obtain an exact equation relating zeros to primes as follows.
Consider zeros $\rho_n = 1/2 + it_n$ of  non-principal and primitive
Dirichlet $L$-functions.
In \cite{FrancaLeclair} we derived the following transcendental equation
on the critical line without assuming the Riemann Hypothesis:
\beq\label{trans_dir}
\vartheta_{k,a}\(t_n\) +
\lim_{\delta\to0^{+}} \Im \log \(
\dfrac{L\(\tfrac{1}{2}+\delta + 
it_n,\chi\)}{L\(\tfrac{1}{2}+\delta, \chi\)} \)
= \(n-\tfrac{1}{2}\)\pi
\eeq
where $n=1,2,\dotsc$ and $t_n > 0$.
Above $\vartheta_{k,a}(t) = 
\Im\log\Gamma\(\tfrac{1}{4}+\tfrac{a}{2} + i\tfrac{t}{2}\) - 
\tfrac{t}{2} \log \tfrac{\pi}{k}$, where $a = 1$ if $\chi(-1)=-1$
and $a=0$ if $\chi(-1)=1$.
The above equation is actually also valid for
$\zeta(s)$ and other principal  $L$-functions.
We argued in \cite{FrancaLeclair}  that if the above equation has
a solution for every $n$, then it saturates the known counting formula on
the entire strip, implying that all zeros must be on
the critical line.  There is a unique solution for every $n$ if one
ignores the $\arg L$ term,  which can be expressed in terms of
the Lambert-$W$ function to a very good approximation.
However, we were not able to justify that this equation does have
a solution for every $n$ with the $\arg L$ term.
The issue is  whether the $\delta$ limit
in \eqref{trans_dir} is well-defined.  Assuming Theorem~\ref{DirConvergence},
it should be well-defined  since the Euler product converges to the
right of the critical line. We therefore have
\beq\label{arg_dir}
\lim_{\delta \to 0^+}  \arg L\(\tfrac{1}{2}+\delta + it, \chi\)
= - \lim_{\delta \to 0^{+}} \Im  \sum_p
\log\( 1-\dfrac{\chi(p)}{p^{1/2+\delta+it}} \).
\eeq
This series converges with $\delta > 0$ thus the above limit
is well-defined.
Note that it is well-defined even on a zero $t=t_n$, since
we get arbitrarily close but we never actually touch the critical line.
Moreover, this shows that $ \lim_{\delta \to 0^+}  \arg L\(\tfrac{1}{2}+\delta
+ it, \chi\)  = O(1)$ as long as  one stays to the right
of the critical line. This justifies that \eqref{trans_dir} has a solution
for every $n$,   however does not rigorously prove so.

More interestingly, equation \eqref{trans_dir} together
with \eqref{arg_dir} no longer makes any reference to the $L$-function itself.
Remarkably,   this   shows that every single
zero $\rho_n = 1/2 + it_n$ is determined by all of the primes,  which is
a converse of Riemann's result for $\pi(x)$ as a sum over zeros.
One can actually in practice  solve for zeros using only primes from
equation \eqref{trans_dir}
with the replacement \eqref{arg_dir}.
In Table~\ref{table_zeros} we provide some numerical data for
the character in \eqref{chi7}.
Approximating $\vartheta_{k,a} (t)$ with Stirling's formula,
one can calculate zeros to very high $t$ using the above formulas.
Using this approach,  in \cite{ALzeta}
one of us computed the $10^{100}$-th zeta zero to over $100$ decimal places.
An interesting question which we leave open is how the error scales with $N$.

\begin{table}
\centering
\def\arraystretch{0.7}
\begin{tabular}{@{}lll@{}}
\toprule[1pt]
$N$ & $t_1$ & error (\%) \\
\midrule[.5pt]
$1$   & $5.57869$ & $7.3$ \\
$10$   & $5.24273$ & $0.85$ \\
$10^2$ & $5.20071$ & $0.05$ \\
$10^3$ & $5.19936$ & $0.02$ \\
$10^4$ & $5.19596$ & $0.04$ \\
$10^5$ & $5.19946$ & $0.02$ \\
$10^6$ & $5.19947$ & $0.02$ \\
\bottomrule[1pt]
\end{tabular}
\caption{\label{table_zeros}
The first zero on the upper critical line for
the Dirichlet character $\chi$ modulo $7$
in \eqref{chi7}   calculated only from knowledge of
the first $N$ primes  from equations \eqref{trans_dir}
and \eqref{arg_dir}.   The actual value is $t_1 = 5.198116\dotsm$.
One can see how the accuracy increases with $N$, but rather slowly.}
\end{table}

The same argument holds for cusp forms through \eqref{EPFmod};
see \cite[eq. (57)]{FrancaLeclair}. Although
not discussed in \cite{FrancaLeclair}, it is now clear that this
only applies to cusp forms.
For non-cusp forms the transcendental equation will not have a solution.

%%%%%%%%%%%%%%%%%%%%%%%%%%%%%%%%%%%%%%%%%%%%%%%%%%%%%%%%%%%%%%%%%%%%%%%%%%%%%%%
\section{The case of Riemann  $\zeta$ and principal Dirichlet}
\label{sec:riemann}

In this last section we briefly remark on the case of
principal Dirichlet characters. The Riemann $\zeta$-function corresponds to the
principal character with modulus $k=1$.
Much of the same reasoning of the previous cases apply,  however, with
some subtle new issues that complicate the analysis.
In these cases $L(s,\chi)$ has a simple pole at $s=1$,
therefore the abscissas  of  convergence of the Euler product are
$\sigma_a = \sigma_c = 1$. Thus, the Euler
product formally diverges for $\sigma \le 1$.
Nevertheless, let us revisit the arguments of
Theorem~\ref{DirConvergence}. For principal characters we have
$\theta_{p_n}=0$, thus we should consider $X(s, \chi)$ in
\eqref{PDir} with \eqref{CxDef}
replaced by
\beq \label{BN2}
C_N (t)  =  \sum_{n=1}^N   \cos (t \log p_n ) .
\eeq
Setting  $t = 0$ gives $C_N = N$, and the same arguments imply
convergence of the Euler product for $\sigma > 1$ which is
consistent with what was just stated above.

However, through further analysis of \eqref{BN2} it can be
shown \cite{Kac} that if
we draw $t$ at random from the interval $[T,2T]$ then \eqref{BN2}
obeys a central limit theorem when
$T \to \infty$ and $N\to\infty$, implying $C_N(t) = O(\sqrt{N})$ in
distribution,  again up to logarithms. The law of iterated logarithms
suggests $C_N (t) = O(\sqrt{N \log \log N})$.
In the case
of \eqref{BN2} the central limit  theorem again relies  on the multiplicative
independence of the primes. Furthermore, it is possible to estimate
\eqref{BN2} directly through the prime number theorem \cite{EulerProduct}
\beq
C_N(t) = \int_{2}^{p_{N}} \cos\(t \log x \) \dfrac{d \pi (x)}{dx} \, dx
\sim \int_{2}^{p_{N}} \cos\(t \log x \) \dfrac{dx}{\log x}.
\eeq
The last integral is expressed  in terms of the $\Ei(z)$ function and
asymptotically yields
\beq \label{BnSmooth}
C_N(t) \sim \dfrac{p_{N}}{\log p_{N}} \dfrac{t}{1+t^2} \sin\( t\log p_N \).
\eeq
The growth of $C_N(t)$ is then given roughly by the ratio $N/t$.
This shows that we can still have
$C_N(t) = O(\sqrt{N})$ for $N \le N_c$ with $N_c = O(t^2)$.
As we will explain,  the need for the cutoff $N_c$ is ultimately
attributed to the existence of the pole.

Therefore, for a fixed $t \gg 0$,
the analysis in Theorem~\ref{DirConvergence} is still valid as long
as we stay below the cutoff $N_c$, i.e. we cannot take the limit $N\to\infty$.
This means that in the region $1/2 < \sigma \le 1$ a truncated
product $\prod_{p \le p_N} \(1 - \chi(p) p^{-s}\)^{-1}$, with
$N\le N_c$, is well-behaved and meaningful despite the fact  that the Euler
product itself is formally divergent.
More specifically, the equality \eqref{logsum2} should now read
\beq
\log L(s,\chi) = X_{N}(s, \chi) + O(1) + R_N(s,\chi)
\eeq
where again $O(1)$ denotes higher order terms which are absolutely convergent
for $\sigma > 1/2$, and $R_N(s,\chi)$ is an error due to truncating
$X(s,\chi)$.
To claim something more precise it is necessary to compute $R_N(s,\chi)$
without further assumptions, which is beyond the scope of this paper.
Yet, if one can show that $R_N(s,\chi)$ becomes small for large $N$ and $t$,
then it may be possible to exclude zeros in the region $\sigma > 1/2$ from
this argument.

Based on a result of Titchmarsh \cite{Titchmarsh} which extends the partial
sum of $\partial_s \log \zeta(s)$ into the critical strip, at the cost of
introducing a sum over non-trivial zeros of $\zeta(s)$,
Gonek, Keating, and Hughes \cite{Gonek,Gonek1} proposed a truncated
Euler product into the critical strip. Assuming the Riemann Hypothesis,
Gonek \cite{Gonek1} estimated the truncation error for
$\sigma \ge \tfrac{1}{2} + \tfrac{1}{\log N}$.
Simply borrowing this result we thus expect something not larger than
\beq
\label{errGonek}
R_N(s,\chi) \ll N^{\tfrac{1}{2}-\sigma}\( 
\log t + \dfrac{\log t}{\log N}\) + \dfrac{N}{t^2 \log^2 N} \sim
N^{\tfrac{1}{2}-\sigma}\log t,
\eeq
where in the last step  we assumed $t< N\le t^2$ for both $N$ and $t$ large.
We see that for $\sigma > 1/2$ the error vanishes in the limit of large $t$ if
$N \sim t^2$.\footnote{Subsequent to this work,
an independent estimate of $R_N$ was made in \cite{ALzeta}.}

Let us clarify why the cutoff is related to the existence of a pole
at $s=1$.
First of all,  due to the pole,  $\log L(s, \chi)$
diverges at $s=1$,  and thus the series $X(s,\chi)$ in \eqref{PDir}
cannot converge for
$\sigma < 1$ for any $t$  since regions of convergence are half-planes.
For non-principal Dirichlet we have the series
$\sum_{n} \cos\left( \theta_{p_n} - t \log p_{n} \right) p_n^{-\sigma}$.
Note that even for $t=0$, the terms $\{ \cos\theta_{p_n} \}$ create enough
randomness automatically regularizing the divergent
series $\sum_n p_n^{-\sigma}$ in the region $1/2 < \sigma \le 1$.
Intuitively, it
is the existence of these terms which prevent the existence of poles
on the real line. For principal Dirichlet these terms are abscent, and
we have the series $\sum_{n} \cos\left( t \log p_n \right) p_n^{-\sigma}$.
For $t=0$ this series is divergent on the region $0 < \sigma \le 1$,
and the pole at $s=1$ is a consequence of this since the $L$-function
can be analytically continued elsewhere. On the real line, $t=0$, nothing
can be done about the divergence of this series.
However, for large enough $t$ we can still
use the terms $\{ \cos\left( t \log p_n\right) \}$ as regularizers
at the price of introducing a cutoff $N_c=O(t^2)$.
As explained through \eqref{errGonek}, this regularization becomes arbitrarily
close to the original series $X(s,\chi)$ in the limit $N_c, t \to \infty$.
In other words, although formally divergent, if we stay away from the
pole, $t \to \infty$,
for all practical purposes the Euler product is still
valid on the right-half part of the critical
strip\footnote{The issues discussed in this section were
subsequently studied in more detail in \cite{ALzeta}.}.

If the above ideas
can be shown to be correct in an unconditional and rigorous manner,
then it suggests
the following approach to the Riemann Hypothesis in this case.
It is already known that
the non-trivial zeros are on the critical line up to at
least $t\sim 10^{10}$ based on
numerical work.
One can then use the asymptotic validity of the Euler product formula to
rule out zeros to the right of the critical line at higher $t$.

%%%%%%%%%%%%%%%%%%%%%%%%%%%%%%%%%%%%%%%%%%%%%%%%%%%%%%%%%%%%%%%%%%%%%%%%%%%%%%%
\section{Concluding remarks}

The main goal of this paper was to identify precisely what properties of
the prime numbers are
responsible for the validity of certain generalized Riemann
Hypotheses. We concluded that it is
their pseudo-random behavior,  which is  a consequence of their
multiplicative independence. This  strongly suggests  that
trigonometric sums over
primes of multiplicative functions behave like random walks, and thus
are bounded by the typical $\sqrt{N}$ growth,  which led us to propose
\mbox{Conjectures~\ref{ConjDir}~and~\ref{ConjCusp}}.
These conjectures were the only unproven assumptions of this paper,  although
we provided strong motivating arguments for their validity.

From these random walk properties, we proved in Theorem~\ref{DirConvergence}
that the Euler
product for non-principal Dirichlet $L$-functions converges to the right
of the critical line. In Theorem~\ref{CuspConvergence} we proved
that the same holds for  $L$-functions based on cusp forms, where Deligne's
result \eqref{Deligne} also plays a major role.
This indicates there is some universality
to our approach.
The original Riemann Hypothesis for $\zeta(s)$ corresponds to the trivial
principal Dirichlet character and is thus not subsumed.
However, in the last section we suggested  how to extend these arguments to
such a case, which is more subtle due to the simple pole at $s=1$.

A vast generalization of Hecke's theory of $L$-functions based on modular
forms is the Langlands program \cite{Langlands}.
There, the $L$-functions are those of Artin,
which are based on Galois number field extensions of the rational numbers.
They have Euler products and satisfy functional equations,  like the cases
studied in this paper.
Langlands automorphic forms play the role of modular forms.
There also exists the notion of a cuspidal form.
It would be very interesting to try and extend the ideas in this paper
to study which of these $L$-functions,  if any,   satisfy a Riemann Hypothesis.

\subsection*{Acknowledgments}
\vspace{-1em}
We would like to thank Denis Bernard,   Nelson Carella, Steve Gonek, and
Ghaith Hiary for discussions,
and the hospitality of
ENS and LPTHE in Paris.  GF thanks partial support from CNPq-Brazil.

%\bibliographystyle{unsrt}
%\bibliography{biblio}
%\input{epf_paper.bbl}

\end{document}